\documentclass[final,3p,times]{elsarticle}

\usepackage[utf8]{inputenc}
\usepackage{amsmath}
\usepackage{amsfonts}
\usepackage{amssymb}
\usepackage{amsthm}
\usepackage{graphicx}
\usepackage{easybmat}
\usepackage{pifont} 
\usepackage{enumitem}
\usepackage{float}
\usepackage{url}
\usepackage{subcaption}
\usepackage{bm}

\usepackage{mathtools}
\mathtoolsset{showonlyrefs}

\biboptions{sort&compress}

\usepackage{color}

\newcommand{\barpow}[1]{\left\lfloor#1\right\rceil}
\newcommand{\sign}[1]{\mbox{sign}(#1)}

\def\Red#1{\textcolor{red}{#1}}

\usepackage{ragged2e}

\usepackage{natbib}

\newcommand{\abs}[1]{\left\lvert#1\right\rvert}
\newcommand{\norm}[1]{\left\lvert\left\lvert#1\right\rvert\right\rvert}
\newtheorem{theorem}{Theorem}

\newtheorem{definition}{Definition}
\newtheorem{example}{Example}
\newtheorem{lemma}{Lemma}

\newtheorem{proposition}{Proposition}
\newtheorem{remark}{Remark}

\def\sign{\hskip2pt{\rm sign}\hskip0pt}

\journal{ArXiv}

\begin{document}

\begin{frontmatter}



\title{Enhancing the settling time estimation of a class of fixed-time stable systems
\footnote{\Red{This is the preprint version of the accepted Manuscript: R. Aldana-López, D. Gómez-Gutiérrez, E. Jiménez-Rodríguez, J. D. Sánchez-Torres and Michael Defoort, “Enhancing the settling time estimation of a class of fixed-time stable systems”, International Journal of Robust and Nonlinear Control, 2019, ISSN: 1099-1239. DOI. 10.1002/rnc.4600.
Please cite the publisher's version. For the publisher's version and full citation details see:
\url{http://dx.doi.org/10.1002/rnc.4600}
}}
}



\author[label0]{R.~Aldana-L\'opez}
\ead{rodrigo.aldana.lopez@gmail.com}

\author[label0,label1]{David~Gómez--Gutiérrez\corref{cor1}}
\ead{David.Gomez.G@ieee.org}
\cortext[cor1]{Corresponding Author.}

\author[label2]{E.~Jim\'enez-Rodr\'iguez}
\ead{ejimenezr@gdl.cinvestav.mx}

\author[label3]{J.~D.~S\'anchez-Torres}
\ead{dsanchez@iteso.mx}

\author[label4]{M.~Defoort}
\ead{michael.defoort@univ-valenciennes.fr}

\address[label0]{Multi-agent autonomous systems lab, Intel Labs, Intel Tecnología de M\'exico, Av. del Bosque 1001, Colonia El Bajío, Zapopan, 45019, Jalisco, M\'exico.}
\address[label1]{Tecnologico de Monterrey, Escuela de Ingenier\'ia y Ciencias, Av. General Ram\'on Corona 2514, Zapopan, 45201, Jalisco, M\'exico.}
\address[label2]{CINVESTAV, Unidad Guadalajara, Av. del Bosque 1145, colonia el Bajío, Zapopan , 45019, Jalisco, M\'exico.}
\address[label3]{Research Laboratory on Optimal Design, Devices and Advanced Materials -OPTIMA-, Department of Mathematics and Physics, ITESO, Perif\'erico Sur Manuel G\'omez Mor\'in 8585 C.P. 45604, Tlaquepaque, Jalisco, M\'exico.}
\address[label4]{LAMIH, CNRS UMR 8201, Univ. Valenciennes, Valenciennes 59313, France.}

\begin{abstract}
This paper deals with the convergence time analysis of a class of fixed-time stable systems with the aim to provide a new non-conservative upper bound for its settling time. Our contribution is fourfold. First, we revisit the well-known class of fixed-time stable systems, given in~\cite{Polyakov2012}, while showing the conservatism of the classical upper estimate of the settling time. Second, we provide the smallest constant that uniformly upper bounds the settling time of any trajectory of the system under consideration. Third, introducing a slight modification of the previous class of fixed-time systems, we propose a new predefined-time convergent algorithm where the least upper bound of the settling time is set \textit{a priori} as a parameter of the system. At last, predefined-time controllers for first order and second order systems are introduced. Some simulation results highlight the performance of the proposed scheme in terms of settling time estimation compared to existing methods.
\end{abstract}


\end{frontmatter}

\section{Introduction}
Convergence time is an important performance specification for a controlled system from a practical point of view \cite{lopez2018fixed,dvir2018acceleration}. Indeed, the design of controllers which guarantee predefined-time stability, instead of asymptotic stability, is one of the desired objectives, which appear in many applications such as missile guidance \cite{songtime}, hybrid formation flying \cite{Zuo2018}, group consensus \cite{shang2017fixed}, online differentiators~\cite{Cruz-Zavala2011,Angulo2013,Basin2016}, state observers~\cite{lopez2018fixed}, etc. Furthermore, in the case of switching systems, it is frequently required that the observer (or controller) achieves the stability of the observation error (or tracking error) before the next switching \cite{defoort2011robust,Gomez2015}.      

Lack of uniform boundedness of the settling-time function regardless of the initial conditions causes several restrictions to the practical application of finite time observer/controller \cite{Haimo1986,Bhat2000,Moulay2006}. These restrictions can be relaxed using the fixed-time stability concept. It is an extension of global finite-time stability, and guarantees the convergence (settling) time to be globally uniformly bounded, i.e., the bound does not depend on the initial state of the system \cite{Andrieu2008,Polyakov2012,songtime}. To this end, the class of systems
\begin{equation}
\label{Eq:HomFixedPoly}
\dot{x} = -(\alpha|x|^{p} + \beta|x|^{q})^k \sign{x}, \ \ x(0)=x_0,
\end{equation}
where $x$ is a scalar state variable and the real numbers $\alpha,\beta,p,q,k>0$ are system parameters which satisfy the constraints $kp<1$, and $kq>1$, was proposed in~\cite{Andrieu2008,Polyakov2012} and has been extensively used. Indeed, it represents a wide class of systems which present the fixed-time stability property through homogeneity and Lyapunov analysis frameworks. However, it is still difficult to derive a relatively simple relationship between the system parameters and the upper bound of the settling time \cite{Jimenez-Rodriguez2017,Jimenez-Rodriguez2018}. This yields some difficulties in the tuning of the system parameters to achieve a prescribed-time stabilization (see for instance \cite{Cruz-Zavala2011}).    

The computation of the least upper bound of the settling time is usually not an easy task. Therefore, it is common to propose an upper bound of the settling time as an attempt to estimate the least upper bound. For instance, in~\cite{Polyakov2012}, it is shown that for system~\eqref{Eq:HomFixedPoly}, the settling-time function $T(x_0)$ is bounded as 
\begin{equation}
\label{Eq:UpperEstimate}  
T(x_0)\leq \frac{1}{\alpha^k(1-pk)}+\frac{1}{\beta^k(qk-1)}=T_{\text{max}},\ \ \ \forall x_0\in\mathbb{R}. 
\end{equation} 
However, this bound significantly overestimates the least upper one. This overestimation can lead to restrictions for the practical implementation of prescribed-time observer/controller. In this case, the gains will be over-tuned to achieve a prescribed-time stabilization. It may lead to poor performances in terms of control magnitude or robustness against measurement noise for instance.  


Considering the extensive use of the class of systems represented by~\eqref{Eq:HomFixedPoly} and the overestimation exhibited by $T_{\text{max}}$ in~\eqref{Eq:UpperEstimate}, this paper addresses the computation of the least upper bound of the settling-time function for this system and the derivation of a new predefined-time convergent algorithm where the least upper bound of the settling time is set \textit{a priori} as a parameter of the system. Based on this result, new predefined-time controllers for first order and second order systems are introduced to enhance the settling time estimation. In contrast to the common approaches of homogeneity and Lyapunov analysis, the results in this paper are derived using the well-known geometric conditions proposed in~\cite{Haimo1986}.
The overall contribution of this paper is divided into the following four main results:
\begin{enumerate}
    \item The well-known class of fixed-time stable systems~\eqref{Eq:HomFixedPoly} is revisited. For this class of systems, the least upper bound of the settling-time function, \[\gamma=\frac{\Gamma \left(m_p\right) \Gamma \left(m_q\right)}{\alpha^{k}\Gamma (k) (q-p)}\left(\frac{\alpha}{\beta}\right)^{m_p},\] with $m_p=\frac{1-kp}{q-p}$ and $m_q=\frac{kq-1}{q-p}$, is found.
    \item The following new predefined-time convergent system, where the least upper bound of the settling time $T_c$ is set \textit{a priori} as a parameter of the system, is proposed:
    \begin{equation}
    \label{Eq:HomFixed}
    \dot{x} = -\frac{\gamma}{T_c}(\alpha|x|^{p} + \beta|x|^{q})^k \sign{x}, \ \ x(0)=x_0,
    \end{equation}
     where $x$ is a scalar state variable, real numbers $\alpha,\beta,p,q,k>0$ are system parameters which satisfy the constraints $kp<1$ and $kq>1$ and $T_c>0$. Notice that the only difference between \eqref{Eq:HomFixedPoly} and the modified system \eqref{Eq:HomFixed} is the constant gain $\gamma/T_c$. This slight change in the original system represents a considerable improvement in its properties, since the tunable parameter $T_c$ is directly the least upper bound of the convergence time. 
     As a consequence of this desirable feature, we say that the origin of system~\eqref{Eq:HomFixed} is predefined-time stable with (strong) predefined-time $T_c$, a notion formally defined in Section \ref{Sec:Preliminaries}.
    \item The bound given in the formula~\eqref{Eq:UpperEstimate} for system~\eqref{Eq:HomFixed} is shown to be a conservative estimation of the settling time $T_f=\sup_{x_0\in\mathbb{R}}T(x_0)=T_c$. Moreover, letting $\alpha=\varrho$ and $\beta=\frac{1}{\varrho}$, it is shown that even if the least upper bound of the convergence time is $T_c$, the upper estimate~\eqref{Eq:UpperEstimate}, given in~\cite{Polyakov2012} goes to infinity as $\varrho\to+\infty$ and as $\varrho\to0$.
\item New predefined-time controllers for first order and second order scalar systems with matched bounded perturbations are introduced to enhance the settling time estimation.
\end{enumerate}

The rest of the manuscript is organized as follows. In Section~\ref{Sec:Preliminaries}, we introduce the preliminaries on finite-time, fixed-time and predefined-time stability. In Section~\ref{Sec:MainResult}, we present the main result on the least upper bound for the settling time and propose a new strongly predefined-time convergent algorithm where the least upper bound of the settling time is set \textit{a priori} as a parameter of the system. We also present the analysis of how conservative the bound provided in~\cite{Polyakov2012} may result and show some numerical results. In Section \ref{Sec:stabilization}, we apply the previous result to derive new predefined-time controllers for first and second order systems. Finally, in Section~\ref{Sec:Conclusion}, we present some concluding remarks.

\section{Preliminaries and Definitions}
\label{Sec:Preliminaries}
Consider the nonlinear system
\begin{equation}\label{Eq:NonSystemDyn}
    \dot{x}=f(x;\rho), \ \ x(0)=x_0,
\end{equation}
where $x\in\mathbb{R}^n$ is the system state, the vector $\rho\in\mathbb{R}^b$ stands for the system~\eqref{Eq:NonSystemDyn} parameters which are assumed to be constant, i.e., $\dot{\rho}=0$. The function $f:\mathbb{R}^n\rightarrow\mathbb{R}^n$ is assumed to be nonlinear and continuous, and the origin is assumed to be an equilibrium point of system~\eqref{Eq:NonSystemDyn}, so $f(0;\rho)=0$.

Let us first recall some useful definitions and lemma on finite-time, fixed-time and predefined-time stability.

\begin{definition}(Lyapunov stability~\cite[Definition~4.1]{Khalil2002})\label{def:lyapstab}
The origin of system \eqref{Eq:NonSystemDyn} is said to be \textit{Lyapunov stable} if for all $\epsilon>0$, there is $\delta:=\delta(\epsilon)>0$ such that for all $\norm{x_0}<\delta$, any solution $x(t,x_0)$ of~\eqref{Eq:NonSystemDyn} exists for all $t\geq0$, and $\norm{x(t,x_0)}<\epsilon$ for all $t\geq0$.
\end{definition}

\begin{definition}(Finite-time stability~\cite{Bhat2000})\label{def:finite} The origin of \eqref{Eq:NonSystemDyn} is said to be \textit{globally finite-time stable} if it is Lyapunov stable, and for any $x_0\in\mathbb{R}^n$, there exists $0\leq T < +\infty$ such that the solution $x(t,x_0)=0$ for all $t\geq T$. The function $T(x_0)=\inf\{T:x(t,x_0)=0, \forall t\geq T\}$ is called the \textit{settling-time function}.
\end{definition}

\begin{lemma}(Finite-time stability characterization for scalar systems~\cite[Fact~1]{Haimo1986}) \label{lem:haimo} Let $n=1$ in system~\eqref{Eq:NonSystemDyn} (scalar system). The origin is globally finite-time stable if and only if for all $x\in\mathbb{R}\setminus\{0\}$
\begin{itemize}
    \item[\textit{(i)}] $xf(x;\rho)<0$, and
    \item[\textit{(ii)}] $\int_{x}^{0}\frac{\text{d}z}{f(z;\rho)}<+\infty.$
\end{itemize}
\end{lemma}

\begin{remark} A proof of Lemma~\ref{lem:haimo} shall not be given here, but can be found in \cite[Lemma~3.1]{Moulay2008}. Nevertheless, intuitively, condition \textit{(i)} implies Lyapunov stability. Moreover, under the conditions of Lemma~\ref{lem:haimo}, note that the settling time function is $T(x_0)=\int_{0}^{T(x_0)}\text{d}t$. Since first-order systems do not oscillate, the solution $x(\cdot,x_0):[0,T(x_0))\to[x_0,0)$ of system~\eqref{Eq:NonSystemDyn} as a function of $t$ defines a bijection. Using it as a variable change, the above integral equals (note that $\frac{1}{f(x;\rho)}$ is defined for all $x\in\mathbb{R}\setminus\{0\}$ from condition \textit{(i)}
\begin{equation}\label{eq:settling_scalar}
    T(x_0)=\int_{0}^{T(x_0)}\text{d}t=\int_{x_0}^{0}\frac{\text{d}x}{f(x;\rho)}.
\end{equation} 
Thus, condition \textit{(ii)} of Lemma~\ref{lem:haimo} refers to the settling-time function being finite.
\end{remark}

\begin{definition}(Fixed-time stability~\cite{Polyakov2012}) \label{def:fixed} The origin is said to be a \textit{fixed-time stable equilibrium} of~\eqref{Eq:NonSystemDyn} if it is globally finite-time-stable and the settling time function $T(x_0)$ is bounded on $\mathbb{R}^n$, i.e. $\exists T_{\text{max}}>0:\forall x_0\in\mathbb{R}^n:T(x_0)\leq T_{\text{max}}$.
\end{definition}

\begin{remark} Let the origin $x=0$ of  system~\eqref{Eq:NonSystemDyn} be fixed-time stable. Notice that there are multiple upper bounds of the settling-time function $T_\text{max}$; for instance, if $T(x_0)\leq T_\text{max}$, also  $T(x_0)\leq \lambda T_\text{max}$ with $\lambda\geq 1$. However, from this boundedness condition, the least upper bound of the settling-time function $\sup_{x_0 \in \mathbb{R}^n} T(x_0)$ exists.
\end{remark}


\begin{remark}
\label{Rem:Predef}
It has been shown that fixed-time stability is guaranteed if the vector field of the system is homogeneous in the bi-limit, a concept defined in~\cite{Andrieu2008,Polyakov2016}. However, in these cases, the upper bound for the settling time is usually not obtained. To differentiate this case to one where a settling time bound $T_c$ is set in advance as a function of system parameters $\rho$, i.e. $T_c=T_c(\rho)$, we introduce the concept of predefined-time stability. A strong notion of this class of stability is given when $\sup_{x_0 \in \mathbb{R}^n} T(x_0)=T_c$, i.e., $T_c$ is the least upper bound for the settling time.
\end{remark}

\begin{definition}(Predefined-time stability~\cite{Sanchez-Torres2018})\label{def:predefined} For the parameter vector $\rho$ of system~\eqref{Eq:NonSystemDyn} and an arbitrarily selected constant $T_c:=T_c(\rho)>0$, the origin of~\eqref{Eq:NonSystemDyn} is said to be \textit{predefined-time stable} if it is fixed-time stable and the settling-time function $T:\mathbb{R}^n\rightarrow\mathbb{R}$ is such that \[T(x_0)\leq T_c, \quad \forall x_0\in\mathbb{R}^n.\] If this is the case, $T_c$ is called a \textit{predefined time}. Moreover, if the settling-time function is such that $\sup_{x_0 \in \mathbb{R}^n} T(x_0)=T_c$, then $T_c$ is called the \textit{strong predefined time}.
\end{definition}

\begin{remark} The stability property, of any kind, refers to equilibrium points of a system. However, since this study only focuses on the global stability of the origin of the system under consideration, it may be referred hereafter, without ambiguity, to the stability of the system in the respective sense (asymptotic, fixed-time or predefined-time).
\end{remark}



\section{On the least upper bound for the settling time of a class of fixed-time stable systems}
\label{Sec:MainResult}

\subsection{Least upper bound of the settling-time function of system~\eqref{Eq:HomFixedPoly}}
Let us first revisit the well-known class of fixed-time stable systems. In the following theorem, based on an appropriate use of the Gamma function, the least upper bound of the settling-time function of  system~\eqref{Eq:HomFixedPoly} is provided.

\begin{theorem}
\label{th:tf_poly}
Let
\begin{equation}
\label{Eq:MinUpperEstimate}
\gamma=\frac{\Gamma \left(m_p\right) \Gamma \left(m_q\right)}{\alpha^{k}\Gamma (k) (q-p)}\left(\frac{\alpha}{\beta}\right)^{m_p},    
\end{equation}
where $\Gamma(\cdot)$ is the Gamma function defined as $\Gamma(z)=\int_0^{+\infty} e^{-t}t^{z-1}dt$~\cite[Chapter~1]{Bateman1953}, and $m_p=\frac{1-k p}{q-p}$ and $m_q=\frac{k q-1}{q-p}$ are positive parameters. The origin $x=0$ of system~\eqref{Eq:HomFixedPoly} is fixed-time stable and the settling time function satisfies $\sup_{x_0 \in \mathbb{R}^n} T(x_0)=\gamma$.
\end{theorem}
\begin{proof} Note that for system~\eqref{Eq:HomFixedPoly}, the field is $f(x;\rho)=-\left(\alpha|x|^{p} + \beta|x|^{q}\right)^k \sign{x}$, where the parameter vector is $\rho=\left[\alpha \, \beta \, p \, q \, k\right]^T\in\mathbb{R}^5$. Furthermore, the product $xf(x;\rho)=- \left(\alpha|x|^{p} + \beta|x|^{q}\right)^k \abs{x}<0$ for all $x\in\mathbb{R}\setminus\{0\}$. Thus, $V(x)=\frac{1}{2}x^2$ is a radially unbounded Lyapunov function for system~\eqref{Eq:HomFixedPoly}, so its origin $x=0$ is Lyapunov stable~\cite[Theorem~4.2]{Khalil2002}.

Now, let $x_0\in\mathbb{R}\setminus\{0\}$ (if $x_0=0$, then $x(t,0)=0$ is the unique solution of~\eqref{Eq:HomFixedPoly} and $T(0)=0$). From~\eqref{eq:settling_scalar}, the settling time function is
\begin{align*}
T(x_0)  &=\int_{x_0}^{0}\frac{\text{d}x}{f(x;\rho)}\\
        &=\int_{0}^{x_0}\frac{\sign{x}\text{d}x}{(\alpha\abs{x}^{p} + \beta\abs{x}^{q})^k}, \qquad z=\abs{x}\\
        &=\int_{0}^{\abs{x_0}}\frac{\text{d}z}{(\alpha z^{p} + \beta z^{q})^k}.
\end{align*}
Since the integrand $\frac{1}{(\alpha z^{p} + \beta z^{q})^k}$ is positive for $z\in\left(0,\abs{x_0}\right)$, the settling time function is increasing with respect to $\abs{x_0}$. Hence, the least upper bound of $T(x_0)$ (in the extended real numbers set, since we do not know yet if it is finite or not) is obtained using Proposition~\ref{prop:integral}, in the Appendix, as \[\sup_{x_0\in\mathbb{R}}T(x_0) = \lim_{|x_0|\to+\infty} T(x_0) = \int_0^{+\infty} \frac{\text{d}z}{(\alpha z^p + \beta z^q)^k}=\gamma,\]
with $\gamma<+\infty$ as in~\eqref{Eq:MinUpperEstimate}. Using Lemma~\ref{lem:haimo} and by the definitions of finite-time and fixed-time stability, the origin $x=0$ of system~\eqref{Eq:HomFixedPoly} is fixed-time stable and $T_f=\gamma$, which completes the proof.
\end{proof}

\subsection{A class of predefined-time stable systems}
Now, the result presented in Theorem~\ref{th:tf_poly} is used to derive a Lyapunov-like condition for characterizing predefined-time stability of a system.

\begin{theorem}(A Lyapunov characterization for predefined-time stable systems)\label{thm:weak_pt}
If there exists a continuous positive definite radially unbounded function $V:\mathbb{R}^n\to\mathbb{R}$ such that any solution $x(t,x_0)$ of~\eqref{Eq:NonSystemDyn} satisfies
\begin{equation}\label{eq:dV_weak}
\dot{V}(x)\leq-\frac{\gamma}{T_c}\left(\alpha V(x)^p+\beta V(x)^q\right)^k, \qquad \forall x\in\mathbb{R}^n\setminus\{0\}.
\end{equation}  
where $\alpha,\beta,p,q,k>0$, $kp<1$, $kq>1$ and $\gamma$ is given in~\eqref{Eq:MinUpperEstimate}. 

Then, the origin of \eqref{Eq:NonSystemDyn} is predefined-time stable and $T_c$ is a predefined time. If in addition, the equality holds in~\eqref{eq:dV_weak}, then $T_c$ is the strong predefined-time.
\end{theorem}
\begin{proof} The negative definiteness of the time derivative of the function $V$ implies Lyapunov stability of the origin of~\eqref{Eq:NonSystemDyn}. Now, suppose that there exists a function $w(t)\geq 0$ that satisfies
\begin{equation*}
    \dot{w} = -(\hat{\alpha} w^p + \hat{\beta} w^q)^k,
\end{equation*}
where $\hat{\alpha} = \alpha\left(\frac{\gamma(\rho)}{T_c}\right)^\frac{1}{k}$ and $\hat{\beta} = \beta\left(\frac{\gamma(\rho)}{T_c}\right)^\frac{1}{k}$, and $V(x_0)\leq w(0)$. Hence, by Theorem \ref{th:tf_poly}, $w(t)$ will converge to the origin in a strong predefined time 
\begin{equation*}
\frac{\Gamma \left(m_p\right) \Gamma \left(m_q\right)}{\hat{\alpha}^{k}\Gamma (k) (q-p)}\left(\frac{\hat{\alpha}}{\hat{\beta}}\right)^{m_p}=\frac{T_c}{\gamma(\rho)}\frac{  \Gamma \left(m_p\right) \Gamma \left(m_q\right)}{\alpha^{k}\Gamma (k) (q-p)}\left(\frac{\alpha}{\beta}\right)^{m_p} = T_c
\end{equation*}
which is directly a tunable parameter of the system. Furthermore, by the comparison lemma~\cite[Lemma~3.4]{Khalil2002}, it follows that $V(x(t))\leq w(t)$, with equality only if \eqref{eq:dV_weak} is an equality. Consequently, the origin of system~\eqref{Eq:NonSystemDyn} is predefined-time stable with predefined time $T_c$. Moreover, if \eqref{eq:dV_weak} is an equality, then $\sup_{x_0\in\mathbb{R}^n}T(x_0)=T_c$, i.e., $T_c$ is the strong predefined time. 
\end{proof}

\begin{example} Consider system~\eqref{Eq:HomFixed} and the continuous positive definite radially unbounded Lyapunov candidate function $V(x)=\abs{x}$ for this system. The derivative of $V(x)$ along the trajectories of system~\eqref{Eq:HomFixed} is \[\dot{V}(x)=-\sign{x}\frac{\gamma}{T_c}\left(\alpha\abs{x}^p+\beta\abs{x}^q\right)^k\sign{x}=-\frac{\gamma}{T_c}\left(\alpha V(x)^p+\beta V(x)^q\right)^k.\] Hence, by Theorem~\ref{thm:weak_pt}, the origin of system~\eqref{Eq:HomFixed} is predefined-time stable with strong predefined time $T_c$.

To illustrate the above, some numerical simulations of system~\eqref{Eq:HomFixed} are conducted, setting the parameters to $\alpha=4$, $\beta=\frac{1}{4}$, $T_c=1$, $p=0.5$, $q=3$, $k=1.5$. The simulations are conducted for several initial conditions $x_0$, as presented in Figure~\ref{Fig:Simu}. It can be seen that $\sup_{x_0\in\mathbb{R}}T(x_0)=T_c=1$ as stated above.
\begin{figure}
    \centering
    \def\svgwidth{13cm}
\begingroup%
  \makeatletter%
  \providecommand\color[2][]{%
    \errmessage{(Inkscape) Color is used for the text in Inkscape, but the package 'color.sty' is not loaded}%
    \renewcommand\color[2][]{}%
  }%
  \providecommand\transparent[1]{%
    \errmessage{(Inkscape) Transparency is used (non-zero) for the text in Inkscape, but the package 'transparent.sty' is not loaded}%
    \renewcommand\transparent[1]{}%
  }%
  \providecommand\rotatebox[2]{#2}%
  \newcommand*\fsize{\dimexpr\f@size pt\relax}%
  \newcommand*\lineheight[1]{\fontsize{\fsize}{#1\fsize}\selectfont}%
  \ifx\svgwidth\undefined%
    \setlength{\unitlength}{370.68019867bp}%
    \ifx\svgscale\undefined%
      \relax%
    \else%
      \setlength{\unitlength}{\unitlength * \real{\svgscale}}%
    \fi%
  \else%
    \setlength{\unitlength}{\svgwidth}%
  \fi%
  \global\let\svgwidth\undefined%
  \global\let\svgscale\undefined%
  \makeatother%
  \begin{picture}(1,0.80045019)%
    \lineheight{1}%
    \setlength\tabcolsep{0pt}%
    \put(0,0){\includegraphics[width=\unitlength,page=1]{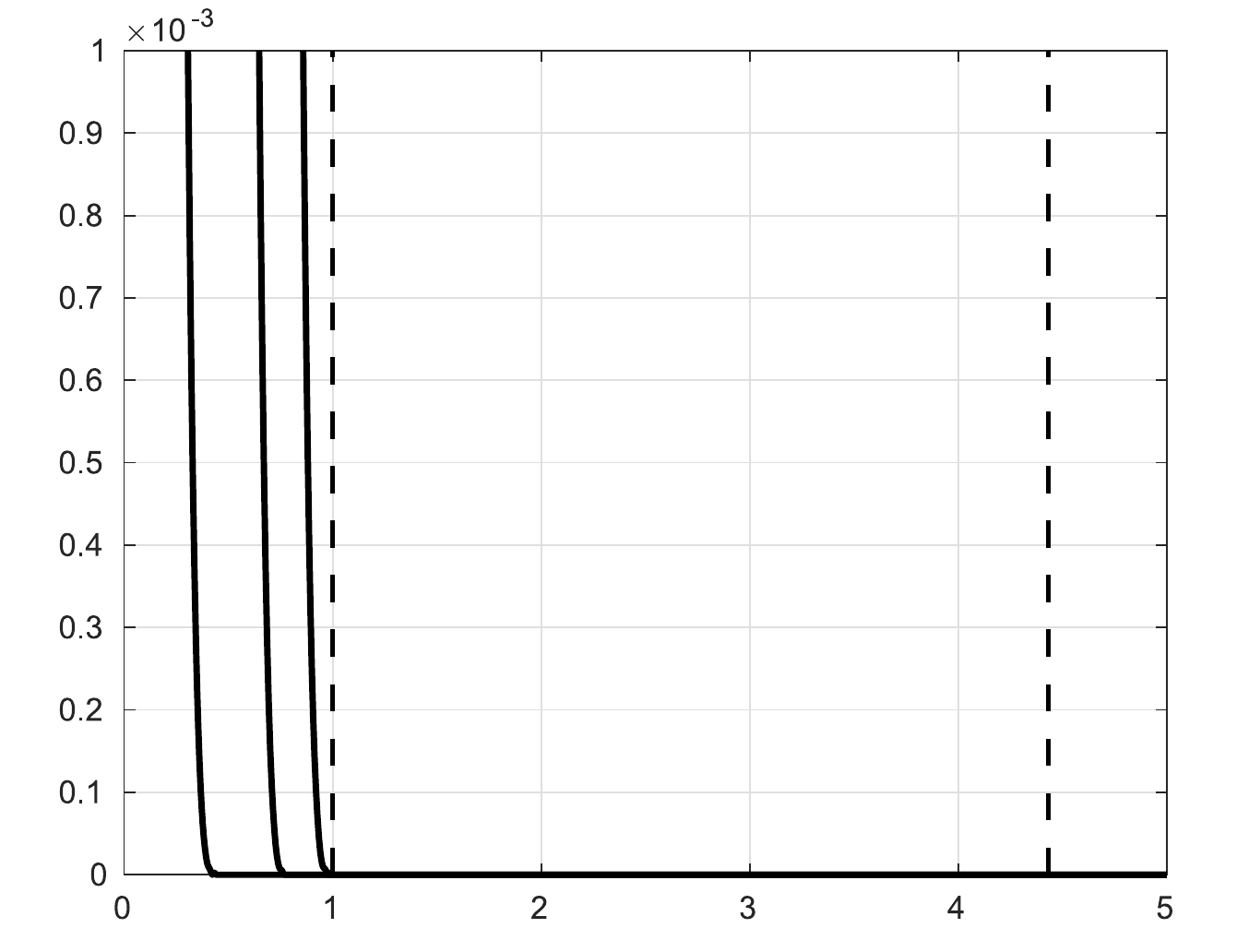}}%
    \put(0.45,0.00434171){\color[rgb]{0,0,0}\makebox(0,0)[lt]{\lineheight{1.25}\smash{\begin{tabular}[t]{l}time $(t)$\end{tabular}}}}%
    \put(0.01660798,0.36178063){\color[rgb]{0,0,0}\rotatebox{90}{\makebox(0,0)[lt]{\lineheight{1.25}\smash{\begin{tabular}[t]{l}$x(t,x_0)$\end{tabular}}}}}%
    \put(0.31,0.36178063){\color[rgb]{0,0,0}\rotatebox{90}{\makebox(0,0)[lt]{\lineheight{1.25}\smash{\begin{tabular}[t]{l}$T_c=1$\end{tabular}}}}}%
    \put(0.9,0.36178063){\color[rgb]{0,0,0}\rotatebox{90}{\makebox(0,0)[lt]{\lineheight{1.25}\smash{\begin{tabular}[t]{l}$T_{\text{max}}(4)=4.4331$\end{tabular}}}}}%
    \put(0.14,0.51){\color[rgb]{0,0,0}\rotatebox{90}{\makebox(0,0)[lt]{\lineheight{1.25}\smash{\begin{tabular}[t]{l}$x_0=0.1$\end{tabular}}}}}%
    \put(0.24,0.51){\color[rgb]{0,0,0}\rotatebox{90}{\makebox(0,0)[lt]{\lineheight{1.25}\smash{\begin{tabular}[t]{l}$x_0=1e{20}$\end{tabular}}}}}%
    \put(0.19,0.51){\color[rgb]{0,0,0}\rotatebox{90}{\makebox(0,0)[lt]{\lineheight{1.25}\smash{\begin{tabular}[t]{l}$x_0=1$\end{tabular}}}}}%
  \end{picture}%
\endgroup%
    \caption{Trajectories of~\eqref{Eq:HomFixed} for different initial conditions with $\alpha=4$, $\beta=\frac{1}{4}$, $T_c=1$, $p=0.5$, $q=3$, $k=1.5$. The upper estimate in~\eqref{Eq:UpperEstimateRho} is $T_{\text{max}}(4)=4.4331s$. The least upper estimate is $T_c=1s$.}
    \label{Fig:Simu}
\end{figure}
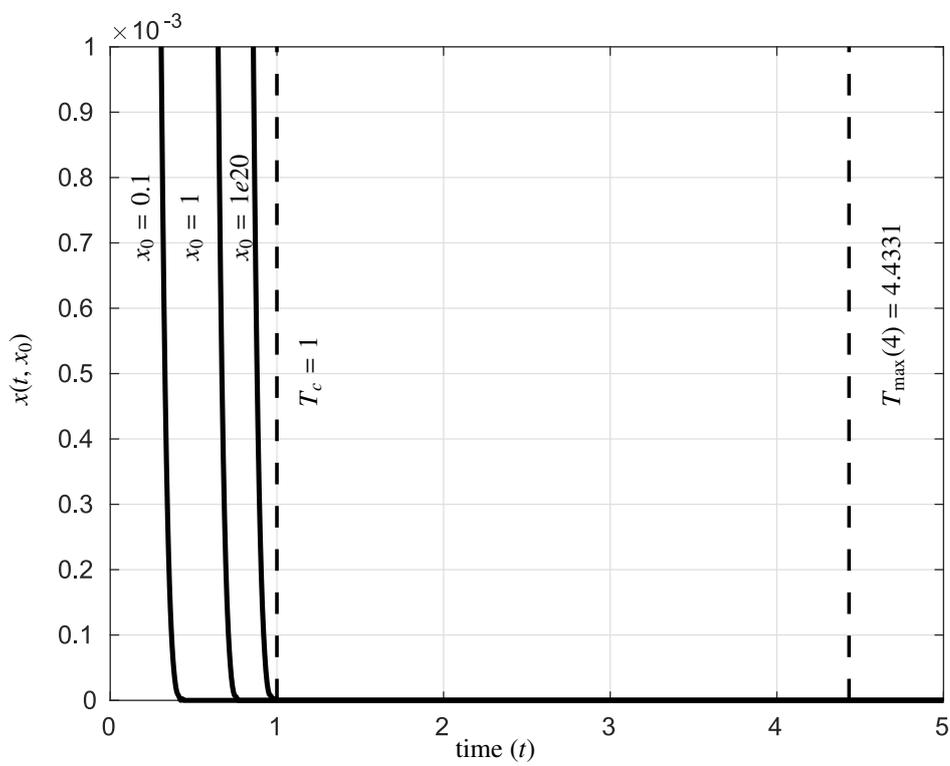


\end{example}

\begin{remark} In~\cite{Parsegov2012}, the least upper estimation of the settling time of~\eqref{Eq:HomFixedPoly} was addressed for the case where $k=1$, $p=1-s$, $q=1+s$, with $0<s<1$, where $\gamma(\rho)$ reduces to $\gamma(\rho)=\frac{\Gamma(\frac{1}{2})^2}{2s\sqrt{\alpha\beta}}=\frac{\pi}{2s\sqrt{\alpha\beta}}$. In~\cite{Jimenez-Rodriguez2018a}, it was shown that, in the case $\alpha=\beta=\frac{\pi}{2sT_c}$, the least upper bound of the settling time is $T_c$. Thus, Theorem~\ref{thm:weak_pt} is a generalization of the results presented in~\cite{Parsegov2012} and~\cite{Jimenez-Rodriguez2018a}. Since only for the case where $k=1$, $p=1-s$ and $q=1+s$, with $0<s<1$, a non-conservative upper bound estimate is provided, many applications, for instance fixed-time consensus protocols~\cite{Zuo2014,Defoort2015,Ning2017b}, have been restricted to this case. 
\end{remark}

Given the relevance of fixed-time stability argumented in the introduction, inequality \eqref{eq:dV_weak} is a result of paramount importance, since as we show in the following, the upper estimate~\eqref{Eq:UpperEstimate} is often too conservative. Thus, applications based on the upper estimate of the settling time~\eqref{Eq:UpperEstimate} presented in~\cite[Lemma~1]{Polyakov2012} are often over-engineered.

\subsection{Settling time bound analysis and comparison}
\label{Sec:TmaxAnalysis}
Consider the predefined-time stable system~\eqref{Eq:HomFixed}, with strong predefined-time $T_c$. From~\eqref{Eq:UpperEstimate}, calculated in~\cite{Polyakov2012}, an upper bound for the settling time $T(x_0)$ is
\begin{equation}
\label{Eq:UpperEstimateModified}  
T(x_0)\leq \frac{T_c}{\gamma(\rho)}\left(\frac{1}{\alpha^k(1-pk)}+\frac{1}{\beta^k(qk-1)}\right),\ \ \ \forall x_0\in\mathbb{R}. 
\end{equation}

Let $\varrho>0$, $\alpha=\varrho$ and $\beta=\frac{1}{\varrho}$. Assuming that $p$, $q$ and $k$ remain constant, and noticing that $\gamma$ is a function of $\varrho$, it can be seen that varying $\varrho$ the least upper bound of the settling time remains constant and equal to $T_c$. However, the bound~\eqref{Eq:UpperEstimateModified} becomes
\begin{equation}
\label{Eq:UpperEstimateRho}  
T(x_0)\leq T_{\text{max}}(\varrho):=\frac{T_c}{K} \left(\frac{1}{\varrho^{2m_p}(1-pk)}+\frac{\varrho^{2(k-2m_p)}}{(qk-1)}\right),
\end{equation}
where $K=\frac{  \Gamma \left(m_p\right) \Gamma \left(m_q\right) }{\Gamma (k) (q-p)}$.
It is easy to see that 
$$
\lim_{\varrho\to 0}T_{\text{max}}(\varrho)=\lim_{\varrho\to\infty}T_{\text{max}}(\varrho)=+\infty,
$$
i.e., $T_{\text{max}}(\varrho)$ in~\eqref{Eq:UpperEstimateRho} has no upper bound as $\varrho$ increases or is close to zero.

Moreover, the best upper estimate of the bound~\eqref{Eq:UpperEstimateRho} is achieved at $\arg\min_{\varrho>0}T_{\text{max}}(\varrho)=1$, and its value is
\begin{equation*}
\label{Eq:UpperEstimateRhoMin}  
\min_{\varrho>0}T_{\text{max}}(\varrho)=T_{\text{max}}(1)=\frac{T_c}{K} \left(\frac{1}{(1-pk)}+\frac{1}{(qk-1)}\right)>T_c.
\end{equation*}

An illustration of this argument, showing $T_{\text{\text{max}}}(\varrho)$ as a function of $\varrho$, with $T_c=1s$, is presented in Figure~\ref{Fig:Convergence}. Although, by Theorem~\ref{thm:weak_pt} the least upper bound of the settling time is $T_c=1s$, it can be seen that in the best case the bound~\eqref{Eq:UpperEstimateModified} provides an overestimation of $\varepsilon T_c$ with $\varepsilon=\left.\frac{1}{K} \left(\frac{1}{(1-pk)}+\frac{1}{(qk-1)}\right)\right|_{p=0.5,q=3,k=1.5}=1.1249$. Moreover, in Figure~\ref{Fig:Simu}, it can be seen that although the least upper bound of the settling-time function is $T_c=1s$, the upper bound estimation provided by~\eqref{Eq:UpperEstimateRho} is $\left.T_{\text{max}}(\varrho)\right|_{\varrho=4}=4.4331s$.
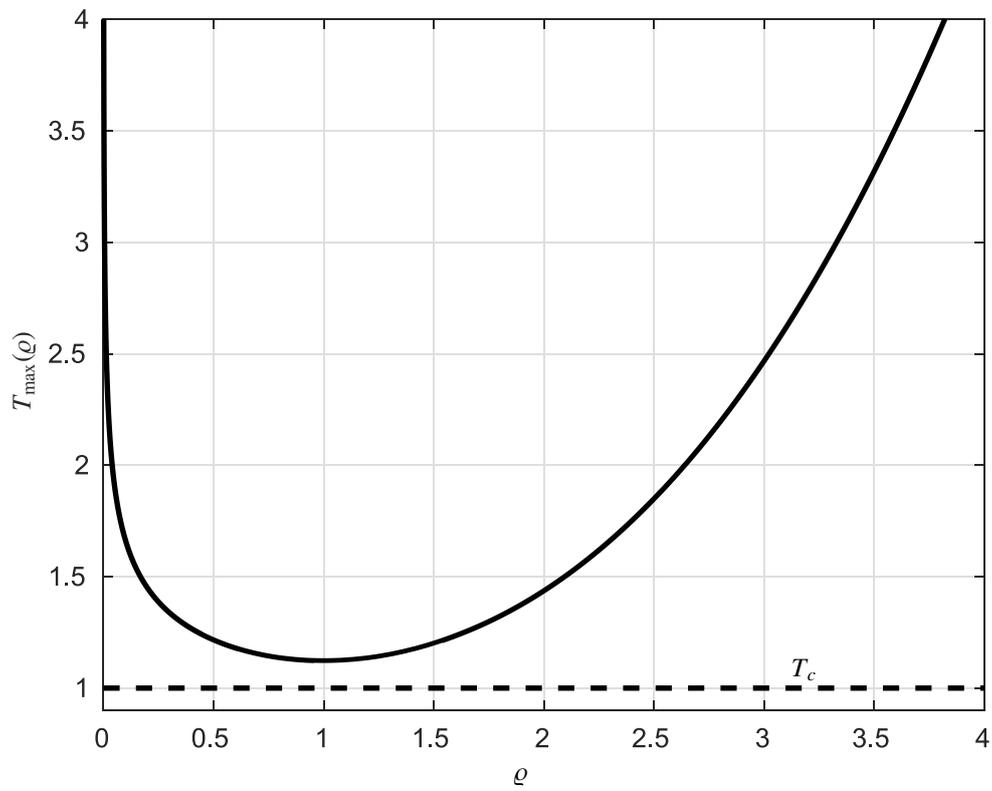
\begin{figure}
    \centering
    \def\svgwidth{13cm}
  \begingroup%
  \makeatletter%
  \providecommand\color[2][]{%
    \errmessage{(Inkscape) Color is used for the text in Inkscape, but the package 'color.sty' is not loaded}%
    \renewcommand\color[2][]{}%
  }%
  \providecommand\transparent[1]{%
    \errmessage{(Inkscape) Transparency is used (non-zero) for the text in Inkscape, but the package 'transparent.sty' is not loaded}%
    \renewcommand\transparent[1]{}%
  }%
  \providecommand\rotatebox[2]{#2}%
  \newcommand*\fsize{\dimexpr\f@size pt\relax}%
  \newcommand*\lineheight[1]{\fontsize{\fsize}{#1\fsize}\selectfont}%
  \ifx\svgwidth\undefined%
    \setlength{\unitlength}{365.05215382bp}%
    \ifx\svgscale\undefined%
      \relax%
    \else%
      \setlength{\unitlength}{\unitlength * \real{\svgscale}}%
    \fi%
  \else%
    \setlength{\unitlength}{\svgwidth}%
  \fi%
  \global\let\svgwidth\undefined%
  \global\let\svgscale\undefined%
  \makeatother%
  \begin{picture}(1,0.78993222)%
    \lineheight{1}%
    \setlength\tabcolsep{0pt}%
    \put(0,0){\includegraphics[width=\unitlength,page=1]{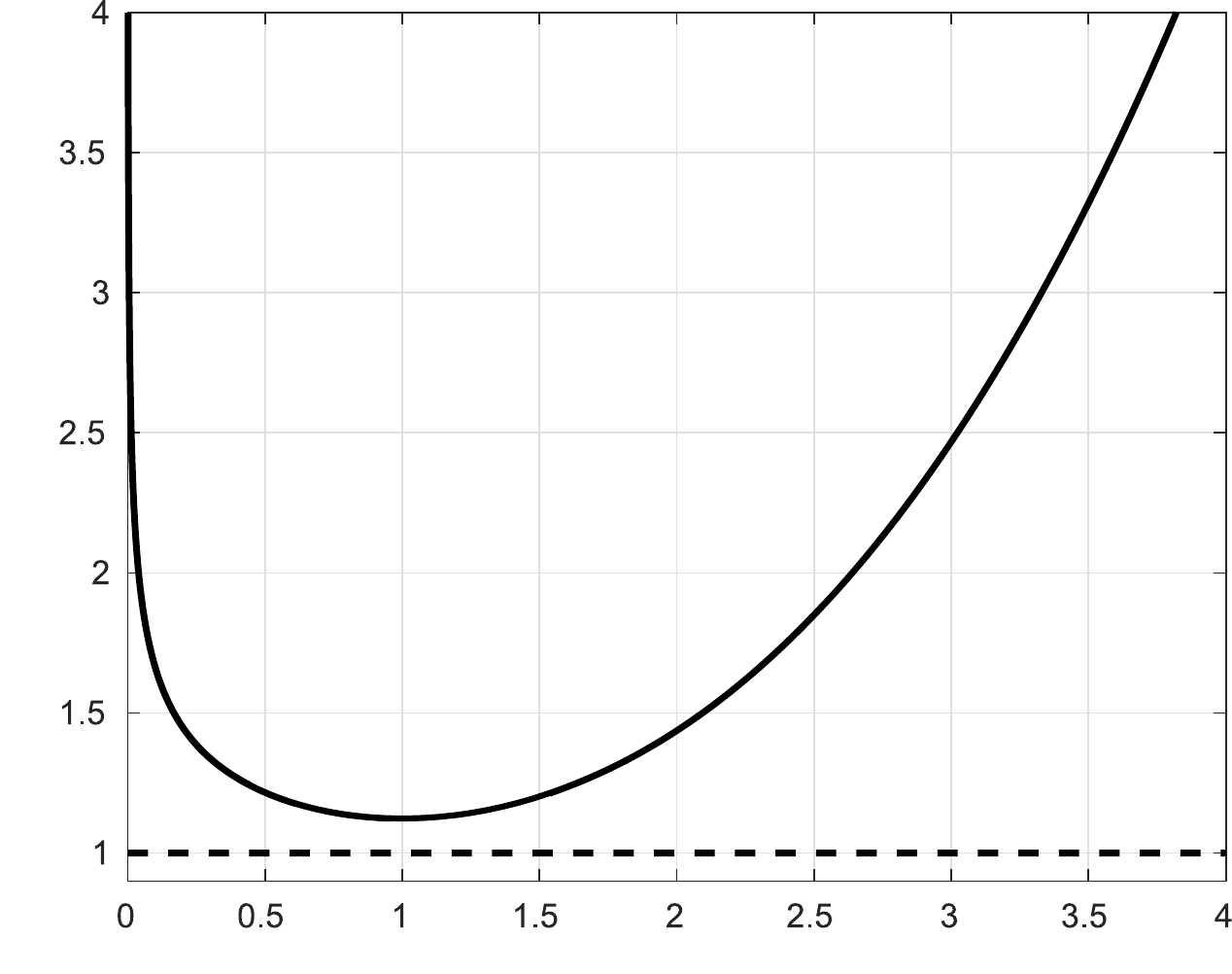}}%
    \put(0.02996607,0.38){\color[rgb]{0.14901961,0.14901961,0.14901961}\rotatebox{90}{\makebox(0,0)[lt]{\lineheight{1.25}\smash{\begin{tabular}[t]{l}$T_{\text{max}}(\varrho)$\end{tabular}}}}}%
    \put(0.52,0.0038629){\color[rgb]{0,0,0}\makebox(0,0)[lt]{\lineheight{1.25}\smash{\begin{tabular}[t]{l}$\varrho$\end{tabular}}}}%
    \put(0.8,0.11){\color[rgb]{0,0,0}\makebox(0,0)[lt]{\lineheight{1.25}\smash{\begin{tabular}[t]{l}$T_c$\end{tabular}}}}%
  \end{picture}%
\endgroup%
    \caption{Function $T_{\text{max}}(\varrho)$ for $p=0.5$, $q=3$, $k=1.5$. The least upper bound of the settling time is $T_c=1s$. The smallest value of $T_{\text{max}}(\varrho)$ is $\min_{\varrho>0}T_{\text{max}}(\varrho)=T_{\text{max}}(1)=1.1249s$.}
    \label{Fig:Convergence}
\end{figure}


\section{Application to robust predefined-time stabilization for first and second order systems}\label{Sec:stabilization}

\subsection{First-order predefined-time controllers}
Consider the following perturbed first-order system
\begin{equation}\label{eq:fosys}
\dot{x}(t)=u(t)+\Delta(t),
\end{equation}
where $x\in\mathbb{R}$ is the state variable, $u\in\mathbb{R}$ is the control input and $\Delta(t)\in\mathbb{R}$ is an unknown but bounded perturbation term of the form $\abs{\Delta(t)}\leq\delta$, with $0\leq\delta<\infty$ a known constant.


The objective is to to design the control input $u$ such that the origin $x=0$ of system~\eqref{eq:fosys} is predefined-time stable, in spite of the unknown perturbation term $\bm{\Delta}(t)$.

\begin{theorem} Let $\alpha,\beta,p,q,k>0$, $kp<1$, $kq>1$, $T_c>0$, $\zeta\geq\delta$ and $\gamma$ be as in~\eqref{Eq:MinUpperEstimate}. If the control input $u$ is selected as
\begin{equation}\label{eq:ufo}
u=-\left[\frac{\gamma}{T_c}\left(\alpha\abs{x}^p+\beta\abs{x}^q\right)^k+\zeta\right]\sign{x}
\end{equation}
then, the origin $x=0$ of system~\eqref{eq:fosys} is predefined-time stable with predefined time $T_c$.
\end{theorem}
\begin{proof} Consider the continuous radially unbounded Lyapunov function candidate $V(x)=\abs{x}$. Its derivative along the trajectories of the closed system~\eqref{eq:fosys}-\eqref{eq:ufo} yields
\begin{align*}
\dot{V}(x)&=-\sign{x}\left[\frac{\gamma}{T_c}\left(\alpha\abs{x}^p+\beta\abs{x}^q\right)^k\sign{x}+\zeta\sign{x}-\Delta\right]\\&=-\frac{\gamma}{T_c}\left(\alpha V(x)^p+\beta V(x)^q\right)^k-\zeta+\Delta\sign{x}\\
&\leq-\frac{\gamma}{T_c}\left(\alpha V(x)^p+\beta V(x)^q\right)^k-\zeta+\abs{\Delta\sign{x}}\\
&\leq-\frac{\gamma}{T_c}\left(\alpha V(x)^p+\beta V(x)^q\right)^k-(\zeta-\delta)\\
&\leq-\frac{\gamma}{T_c}\left(\alpha V(x)^p+\beta V(x)^q\right)^k.
\end{align*}
Hence, using Theorem~\ref{thm:weak_pt}, the origin $x=0$ of system~\eqref{eq:fosys} is predefined-time stable with predefined time $T_c$. 
\end{proof}
\begin{example}
An example of this approach is shown in Figure \ref{Fig:ScalarControl}, where the control input $u$ given in~\eqref{eq:ufo}, with $\zeta=1$, $p=0.5$, $q=3$, $k=1.5$ and $\alpha=1/\beta=\varrho=4$, is applied to the perturbed system~\eqref{eq:fosys} with disturbance $\Delta(t) = \sin(2\pi t /5)$. Note that although the upper bound of the settling-time function is $T_c=1s$ using the proposed scheme, the upper bound estimation provided by~\cite{Polyakov2012} is $T_{\text{max}}(\varrho)|_{\varrho=4} = 4.4331s$.

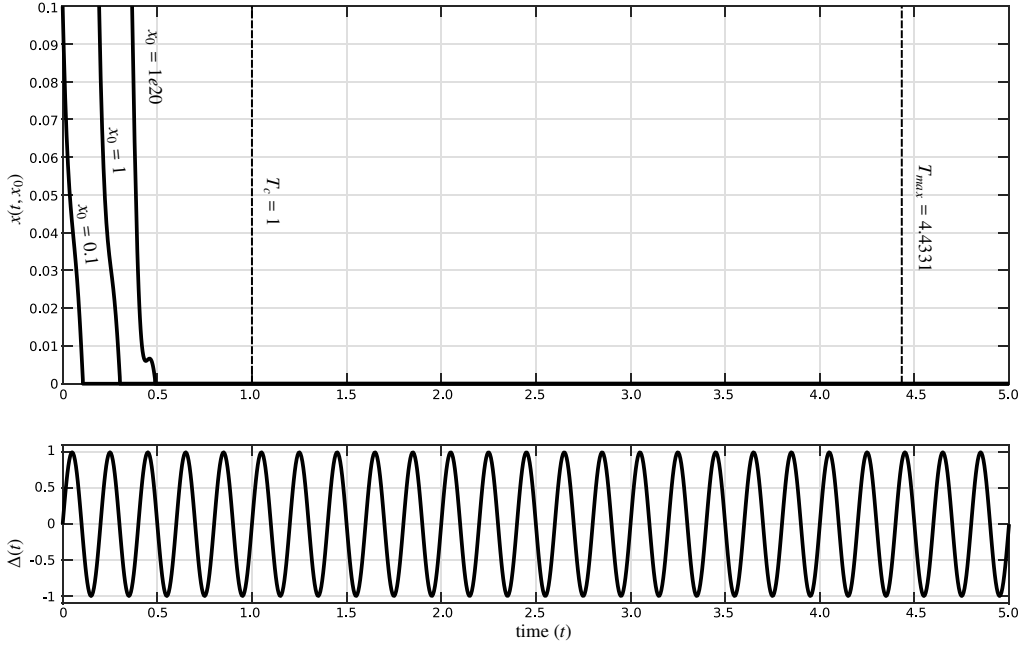
\begin{figure}
\centering
\def\svgwidth{13cm}
\begingroup%
  \makeatletter%
  \providecommand\color[2][]{%
    \errmessage{(Inkscape) Color is used for the text in Inkscape, but the package 'color.sty' is not loaded}%
    \renewcommand\color[2][]{}%
  }%
  \providecommand\transparent[1]{%
    \errmessage{(Inkscape) Transparency is used (non-zero) for the text in Inkscape, but the package 'transparent.sty' is not loaded}%
    \renewcommand\transparent[1]{}%
  }%
  \providecommand\rotatebox[2]{#2}%
  \newcommand*\fsize{\dimexpr\f@size pt\relax}%
  \newcommand*\lineheight[1]{\fontsize{\fsize}{#1\fsize}\selectfont}%
  \ifx\svgwidth\undefined%
    \setlength{\unitlength}{1125.88330078bp}%
    \ifx\svgscale\undefined%
      \relax%
    \else%
      \setlength{\unitlength}{\unitlength * \real{\svgscale}}%
    \fi%
  \else%
    \setlength{\unitlength}{\svgwidth}%
  \fi%
  \global\let\svgwidth\undefined%
  \global\let\svgscale\undefined%
  \makeatother%
  \begin{picture}(1,0.62707645)%
    \lineheight{1}%
    \setlength\tabcolsep{0pt}%
    \put(0,0){\includegraphics[width=\unitlength]{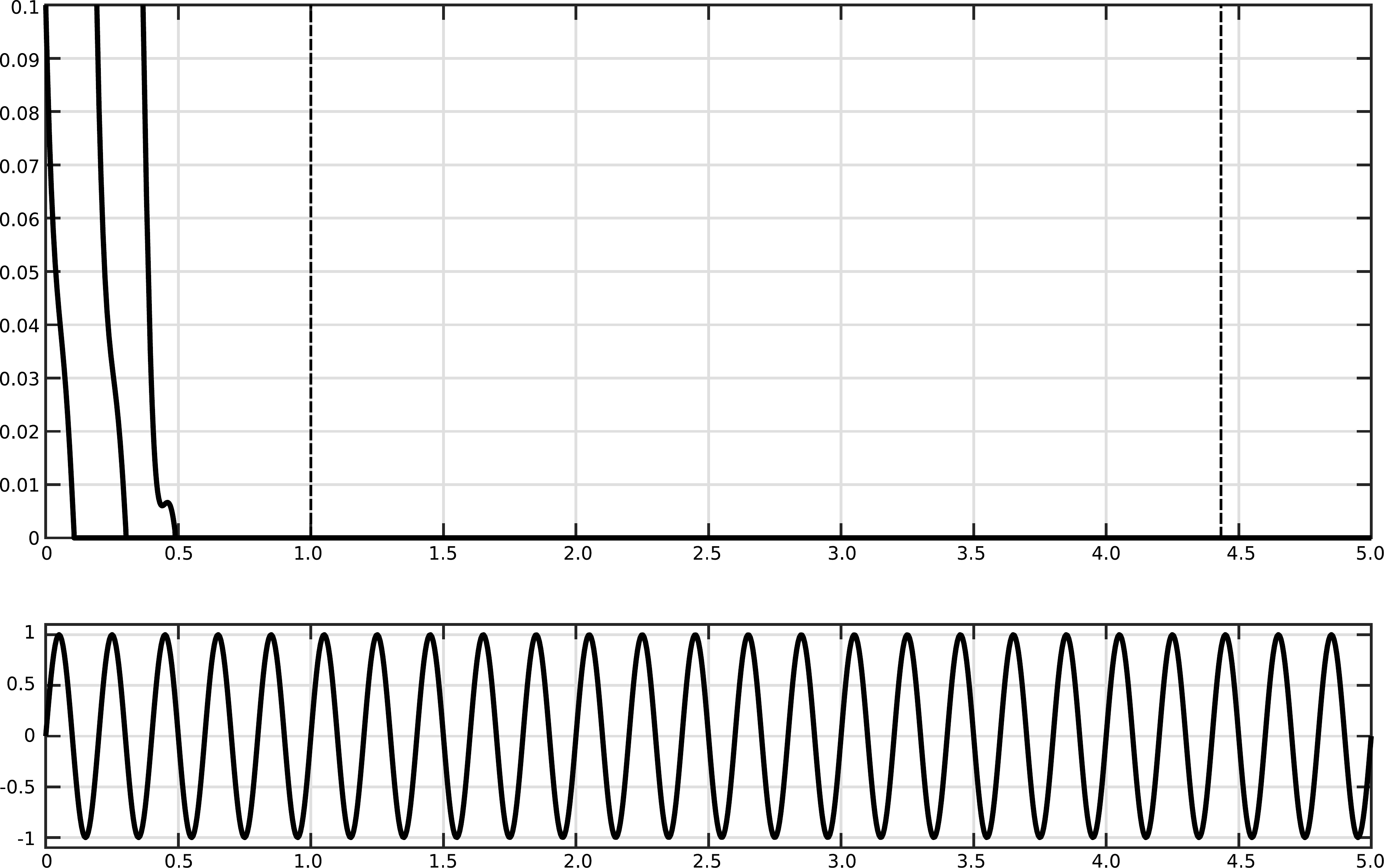}}%
    \scriptsize{
    \put(0.049,0.42163523){\color[rgb]{0,0,0}\rotatebox{-82.70268644}{\makebox(0,0)[lt]{\lineheight{1.25}\smash{\begin{tabular}[t]{l}$x_0=0.1$\end{tabular}}}}}%
    \put(0.49168561,-0.02){\color[rgb]{0,0,0}\makebox(0,0)[lt]{\lineheight{1.25}\smash{\begin{tabular}[t]{l}time $(t)$\end{tabular}}}}%
    \put(0.08,0.5){\color[rgb]{0,0,0}\rotatebox{-84.71961}{\makebox(0,0)[lt]{\lineheight{1.25}\smash{\begin{tabular}[t]{l}$x_0=1$\end{tabular}}}}}%
    \put(0.12,0.6){\color[rgb]{0,0,0}\rotatebox{-89.200675}{\makebox(0,0)[lt]{\lineheight{1.25}\smash{\begin{tabular}[t]{l}$x_0=1e20$\end{tabular}}}}}%
    \put(0.24,0.44934085){\color[rgb]{0,0,0}\rotatebox{-90}{\makebox(0,0)[lt]{\lineheight{1.25}\smash{\begin{tabular}[t]{l}$T_c=1$\end{tabular}}}}}%
    \put(0.9,0.46274433){\color[rgb]{0,0,0}\rotatebox{-90}{\makebox(0,0)[lt]{\lineheight{1.25}\smash{\begin{tabular}[t]{l}$T_{max}=4.4331$\end{tabular}}}}}%
    \put(-0.01,0.4){\color[rgb]{0,0,0}\rotatebox{90}{\makebox(0,0)[lt]{\lineheight{1.25}\smash{\begin{tabular}[t]{l}$x(t,x_0)$\end{tabular}}}}}%
    \put(-0.01,0.05){\color[rgb]{0,0,0}\rotatebox{90}{\makebox(0,0)[lt]{\lineheight{1.25}\smash{\begin{tabular}[t]{l}$\Delta(t)$\end{tabular}}}}}%
    }
  \end{picture}%
\endgroup%
\caption{State trajectory of the closed-loop system~\eqref{eq:fosys}-\eqref{eq:ufo} with $T_c=1s$ and the considered perturbation.}
\label{Fig:ScalarControl}
\end{figure}

\end{example}

\subsection{Second-order predefined-time controllers}
Consider the following perturbed second-order system
\begin{align}\label{eq:sosys}
\begin{split}
\dot{x}_1&=x_2\\
\dot{x}_2&=u+\Delta(t),    
\end{split}    
\end{align}
where $x_1,x_2\in\mathbb{R}$ are the state variables, $u\in\mathbb{R}$ is the control input and $\Delta(t)\in\mathbb{R}$ is an unknown but bounded perturbation term of the form $\abs{\Delta(t)}\leq\delta$, with $0\leq\delta<\infty$ a known constant.


The objective is to design the control input $u$ such that the origin $(x_1,x_2)=(0,0)$ of system~\eqref{eq:sosys} is predefined-time stable, in spite of the unknown perturbation term $\bm{\Delta}(t)$.

The following definition will be useful for stating Theorem~\ref{thm:socont}.

\begin{definition}\label{def:oddpower} For any real number $r$, the function $x\to \barpow{x}^r$ is defined as $\barpow{x}=\abs{x}^r \sign{x}$ for any $x\in \mathbb{R}$ if $r>0$, and for any $x\in \mathbb{R}\setminus\{0\}$ if $r\leq0$. 
\end{definition}

\begin{theorem}\label{thm:socont} Let $\alpha_1,\alpha_2,\beta_1,\beta_2,p,q,k>0$, $kp<1$, $kq>1$, $T_{c_1},T_{c_2}>0$, $\zeta\geq\delta$, and \[\gamma_1=\frac{\Gamma \left(\frac{1}{4}\right)^2 }{2\alpha_1^{1/2}\Gamma\left(\frac{1}{2}\right)}\left(\frac{\alpha_1}{\beta_1}\right)^{1/4},\text{ and } \gamma_2=\frac{\Gamma \left(m_{p}\right) \Gamma \left(m_{q}\right)}{\alpha_2^{k}\Gamma (k) (q-p)}\left(\frac{\alpha_2}{\beta_2}\right)^{m_{p}},\] with $m_{p}=\frac{1-kp}{q-p}$ and $m_{q}=\frac{kq-1}{q-p}$. If the control input is selected as
\begin{equation}\label{eq:uso}
u=-\left[\frac{\gamma_2}{T_{c_2}}\left(\alpha_2\abs{\sigma}^{p}+\beta_2\abs{\sigma}^{q}\right)^{k}+\frac{\gamma_1^2}{2T_{c_1}^2}\left(\alpha_1+3\beta_1x_1^2\right)+\zeta\right]\sign{\sigma},
\end{equation}
where the sliding variable $\sigma$ is defined as
\begin{equation}\label{eq:sigmaso}
\sigma=x_2+\barpow{\barpow{x_2}^2+\frac{\gamma_1^2}{T_{c_1}^2}\left(\alpha_1\barpow{x_1}^1+\beta_1\barpow{x_1}^3\right)}^{1/2},
\end{equation}
then the origin $(x_1,x_2)=(0,0)$ of system~\eqref{eq:sosys} is predefined-time stable with predefined time $T_c=T_{c_1}+T_{c_2}$.
\end{theorem}
\begin{proof} The time-derivative of the sliding variable $\sigma$ in~\eqref{eq:sigmaso} is
\begin{align*}
    \begin{split}
        \dot{\sigma}&=u+\Delta+\frac{\abs{x_2}(u+\Delta)+\frac{\gamma_1^2}{2T_{c_1}^2}\left(\alpha_1+3\beta_1x_1^2\right)x_2}{\abs{\barpow{x_2}^2+\frac{\gamma_1^2}{T_{c_1}^2}\left(\alpha_1\barpow{x_1}^1+\beta_1\barpow{x_1}^3\right)}^{1/2}}\\
        &=-\frac{\gamma_2}{T_{c_2}}\left(\alpha_2\abs{\sigma}^{p}+\beta_2\abs{\sigma}^{q}\right)^{k}\sign{\sigma}-\zeta\sign{\sigma}+\Delta-\frac{\gamma_1^2}{2T_{c_1}^2}\left(\alpha_1+3\beta_1x_1^2\right)\sign{\sigma}\\
        &-\frac{\frac{\gamma_2}{T_{c_2}}\left(\alpha_2\abs{\sigma}^{p}+\beta_2\abs{\sigma}^{q}\right)^{k}\sign{\sigma}+\zeta\sign{\sigma}-\Delta}{\abs{\barpow{x_2}^2+\frac{\gamma_1^2}{T_{c_1}^2}\left(\alpha_1\barpow{x_1}^1+\beta_1\barpow{x_1}^3\right)}^{1/2}}-\frac{\frac{\gamma_1^2}{2T_{c_1}^2}\left(\alpha_1+3\beta_1x_1^2\right)}{\abs{\barpow{x_2}^2+\frac{\gamma_1^2}{T_{c_1}^2}\left(\alpha_1\barpow{x_1}^1+\beta_1\barpow{x_1}^3\right)}^{1/2}}\left(\abs{x_2}\sign{\sigma}-x_2\right).
    \end{split}
\end{align*}

Now, considering $V_2(\sigma)=\abs{\sigma}$ as a continuous radially unbounded positive definite Lyapunov candidate function it can be easily checked using the above that \[\dot{V}(\sigma)\leq-\frac{\gamma_2}{T_{c_2}}\left(\alpha_2 V_2(\sigma)^{p}+\beta_2 V_2(\sigma)^{q}\right)^{k},\]
and using Theorem~\ref{thm:weak_pt}, the origin $\sigma=0$ of the sliding variable dynamics is predefined-time stable with predefined time $T_{c_2}$.

Once the system trajectories are constrained to the manifold $\sigma=0$, i.e. for $t\geq T_{c_2}$, the solutions of system~\eqref{eq:sosys} satisfy the following reduced-order dynamics (see~\eqref{eq:sigmaso}): \[\dot{x}_1=x_2=-\frac{\gamma_1}{T_{c_1}}\left(\alpha_1\abs{x_1}+\beta_1\abs{x_1}^3\right)^{1/2}\sign{x_1}.\]
Thus, considering $V_1(x_1)=\abs{x_1}$ as a continuous radially unbounded positive definite Lyapunov candidate function and using Theorem~\ref{thm:weak_pt}, it is concluded that the origin $x_1=0$ of the reduced order system is predefined-time stable with predefined time $T_{c_1}$. 
Moreover, from~\eqref{eq:sigmaso}, if $\sigma=0$ and $x_1=0$, then $x_2=0$. Hence, it is concluded that the origin $(x_1,x_2)=(0,0)$ of the closed-loop system~\eqref{eq:sosys}-\eqref{eq:uso} is predefined-time stable with predefined time $T_{c_1}+T_{c_2}$.
\end{proof}

\begin{example}
An example of this approach is shown in Figure~\ref{Fig:SecondOrderControl}, where the control input $u$ given in~\eqref{eq:uso}, with $\zeta=1$, $p=0.5$, $q=3$, $k=1.5$ and $\alpha_1=\alpha_2=1/\beta_1=1/\beta_2=4$ is applied to the perturbed system~\eqref{eq:sosys} with disturbance $\Delta(t) = \sin(2\pi t /5)$. It follows from Theorem~\ref{thm:socont} that with $T_{c_1} = T_{c_2} = 0.5s$, the origin $(x_1,x_2)=(0,0)$ of system~\eqref{eq:sosys} is predefined-time stable with predefined time $T_c=1s$. Note that although the upper bound of the settling-time function is $T_c=1s$ using the proposed scheme, the upper bound estimation provided by~\cite{Polyakov2012} is  $T_{\text{max}}=\frac{T_{c_2}}{\gamma_2}\left(\frac{1}{\alpha_{2}^k(1-pk)} + \frac{1}{\beta_{2}^k(qk-1)}\right) + \frac{2T_{c_1}}{\gamma_1}\left(\frac{1}{\sqrt{\alpha_{1}}} + \frac{1}{\sqrt{\beta_{1}}}\right)  =  5.1073s$.
\begin{figure}
\centering
\def\svgwidth{13cm}
\begingroup%
  \makeatletter%
  \providecommand\color[2][]{%
    \errmessage{(Inkscape) Color is used for the text in Inkscape, but the package 'color.sty' is not loaded}%
    \renewcommand\color[2][]{}%
  }%
  \providecommand\transparent[1]{%
    \errmessage{(Inkscape) Transparency is used (non-zero) for the text in Inkscape, but the package 'transparent.sty' is not loaded}%
    \renewcommand\transparent[1]{}%
  }%
  \providecommand\rotatebox[2]{#2}%
  \newcommand*\fsize{\dimexpr\f@size pt\relax}%
  \newcommand*\lineheight[1]{\fontsize{\fsize}{#1\fsize}\selectfont}%
  \ifx\svgwidth\undefined%
    \setlength{\unitlength}{1153.77978516bp}%
    \ifx\svgscale\undefined%
      \relax%
    \else%
      \setlength{\unitlength}{\unitlength * \real{\svgscale}}%
    \fi%
  \else%
    \setlength{\unitlength}{\svgwidth}%
  \fi%
  \global\let\svgwidth\undefined%
  \global\let\svgscale\undefined%
  \makeatother%
  \begin{picture}(1,0.6335442)%
    \lineheight{1}%
    \setlength\tabcolsep{0pt}%
    \footnotesize{
    \put(0,0){\includegraphics[width=\unitlength]{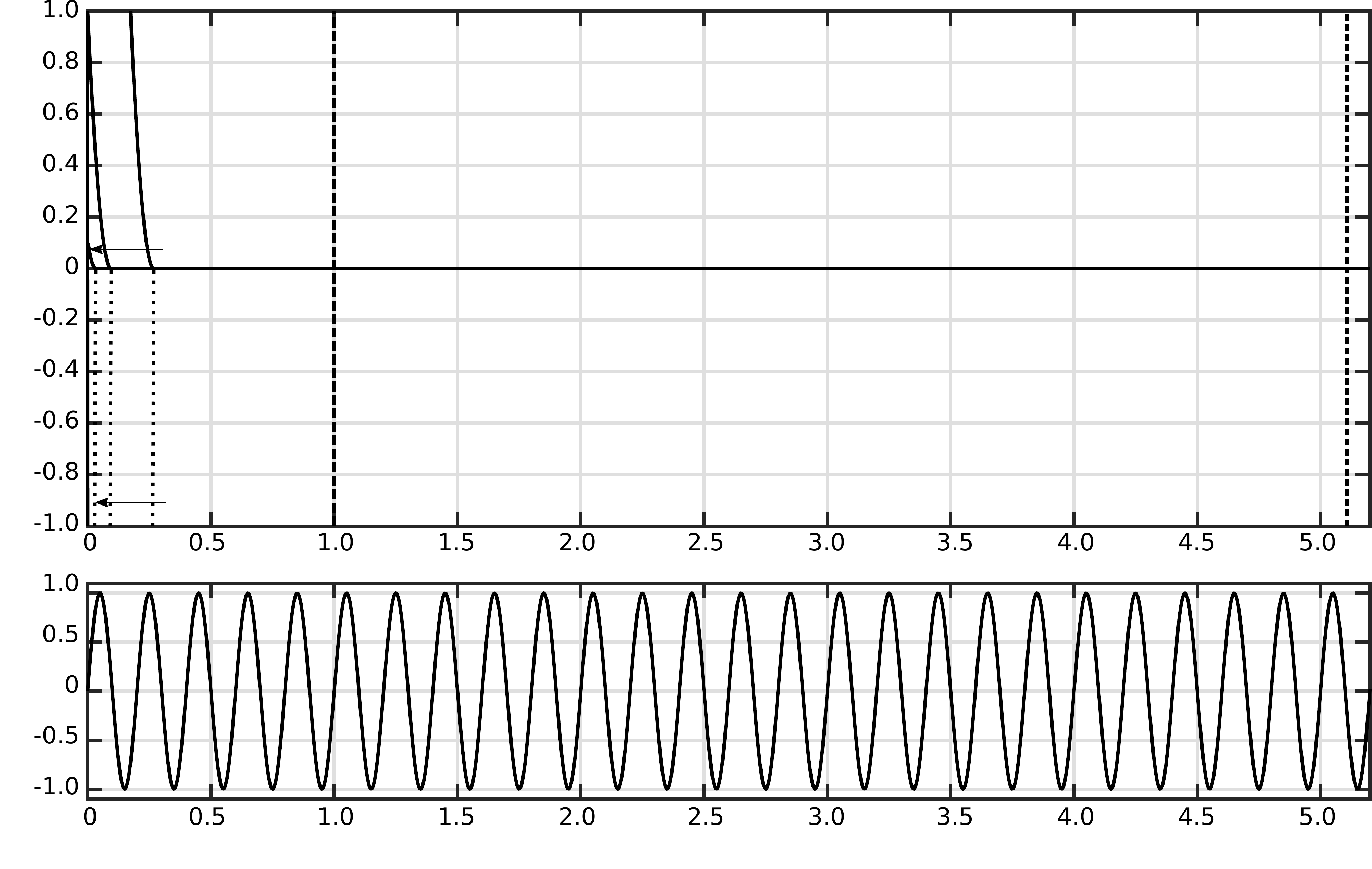}}%
    \put(0.48648,0.00191075){\color[rgb]{0,0,0}\makebox(0,0)[lt]{\lineheight{1.25}\smash{\begin{tabular}[t]{l}time ($t$)\end{tabular}}}}%
    \put(0.0214886,0.10780238){\color[rgb]{0,0,0}\rotatebox{90}{\makebox(0,0)[lt]{\lineheight{1.25}\smash{\begin{tabular}[t]{l}$\Delta(t)$\end{tabular}}}}}%
    \put(0.01130764,0.31267537){\color[rgb]{0,0,0}\rotatebox{90}{\makebox(0,0)[lt]{\lineheight{1.25}\smash{\begin{tabular}[t]{l}$x_2(t,x_0)$\end{tabular}}}}}%
    \put(0.00987752,0.50115304){\color[rgb]{0,0,0}\rotatebox{90}{\makebox(0,0)[lt]{\lineheight{1.25}\smash{\begin{tabular}[t]{l}$x_1(t,x_0)$\end{tabular}}}}}%
    \put(0.128,0.43){\color[rgb]{0,0,0}\rotatebox{-90}{\makebox(0,0)[lt]{\lineheight{1.25}\smash{\begin{tabular}[t]{l}$x_2(0)=100$\end{tabular}}}}}%
    \put(0.09,0.43){\color[rgb]{0,0,0}\rotatebox{-90}{\makebox(0,0)[lt]{\lineheight{1.25}\smash{\begin{tabular}[t]{l}$x_2(0)=1$\end{tabular}}}}}%
    \put(0.125,0.27){\color[rgb]{0,0,0}\makebox(0,0)[lt]{\lineheight{1.25}\smash{\begin{tabular}[t]{l}$x_2(0)=0.1$\end{tabular}}}}%
    \put(0.11,0.62){\color[rgb]{0,0,0}\rotatebox{-88}{\makebox(0,0)[lt]{\lineheight{1.25}\smash{\begin{tabular}[t]{l}$x_1(0)=100$\end{tabular}}}}}%
    \put(0.077,0.62){\color[rgb]{0,0,0}\rotatebox{-88}{\makebox(0,0)[lt]{\lineheight{1.25}\smash{\begin{tabular}[t]{l}$x_1(0)=1$\end{tabular}}}}}%
    \put(0.125,0.45){\color[rgb]{0,0,0}\makebox(0,0)[lt]{\lineheight{1.25}\smash{\begin{tabular}[t]{l}$x_1(0)=0.1$\end{tabular}}}}%
    \put(0.23,0.3){\color[rgb]{0,0,0}\rotatebox{90}{\makebox(0,0)[lt]{\lineheight{1.25}\smash{\begin{tabular}[t]{l}$T_c=1$\end{tabular}}}}}%
    \put(0.97,0.3){\color[rgb]{0,0,0}\rotatebox{90}{\makebox(0,0)[lt]{\lineheight{1.25}\smash{\begin{tabular}[t]{l}$T_{max}=5.1073$\end{tabular}}}}}%
    }
  \end{picture}%
\endgroup%
\caption{State trajectory of the closed-loop system~\eqref{eq:sosys}-\eqref{eq:uso} with $T_c=1s$ and the considered perturbation.}\label{Fig:SecondOrderControl}
\end{figure}
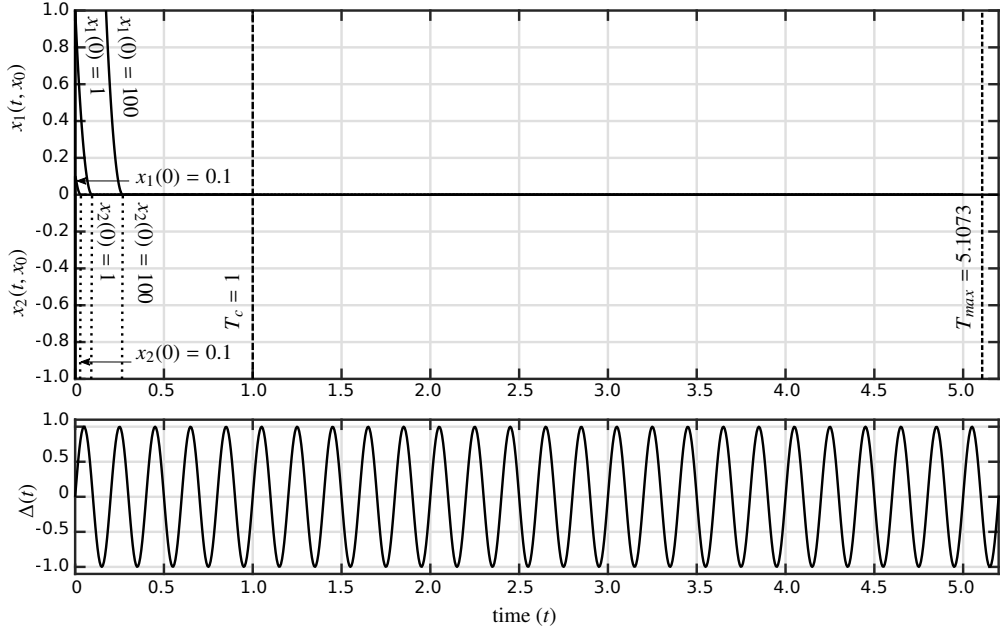
\end{example}

\section{Conclusion}
\label{Sec:Conclusion}
In this paper, we studied the convergence time of a class of fixed-time stable systems with the aim to provide a new non-conservative upper bound for its settling time. We showed that the well-known upper bound condition for the settling time of this class of systems is often too conservative. To illustrate our claim, we showed how by changing one parameter the upper estimate of the settling time tends to infinity even though the actual settling time is always bounded by a constant $T_c$. To address this problem, we proposed a modification to the classical fixed-time algorithm to transform it into a strongly predefined-time (in which the upper bound for the settling time is set in advance as a parameter of the system and is the lowest upper estimate of the settling time) with strong predefined-time $T_c$. With this result, the Lyapunov inequality, which is a sufficient condition for fixed-time stability, was modified in a way that it becomes predefined-time parametrized by $T_c$. When such inequality becomes equality, $T_c$ becomes the lowest upper estimate of the settling time of the system. At last, predefined-time controllers for first order and second order systems were introduced. Some simulation results have shown the performance of the proposed scheme in terms of settling time estimation compared to existing methods. This is an important contribution toward online differentiators, observers and controllers satisfying prescribed-time objectives.
\appendix
\section{Auxiliary Results}
\begin{definition}\cite[Pg.~87]{Bateman1953}\label{def:beta}
Let $a,b>0$. The Beta function, denoted by $B(a,b)$, is defined as
\begin{align*}
    B(a,b)&=\int_0^1z^{a-1}(1-z)^{b-1}\text{d}z=\frac{\Gamma(a)\Gamma(b)}{\Gamma(a+b)}.
\end{align*}
\end{definition}

\begin{proposition}
\label{prop:integral}
Let $\alpha,\beta,p,q,k>0$, with $pk<1$, $qk>1$. Hence,
\begin{equation}
\int_0^{+\infty} \frac{\mbox{\normalfont d} x}{(\alpha x^p + \beta x^q)^k} = \frac{  \Gamma \left(m_p\right) \Gamma \left(m_q\right) \left(\frac{\alpha}{\beta}\right)^{m_p}}{\alpha^{k}\Gamma (k) (q-p)}.
\label{eq:integral_fixed}
\end{equation}
\end{proposition}
\begin{proof}
The left-hand side of \eqref{eq:integral_fixed} can be rewritten as
\begin{equation}
\int_0^{+\infty} \frac{\text{d}x}{(\alpha x^p + \beta x^q)^k}=\int_0^{+\infty} \frac{\beta^{-k}x^{-qk}\text{d}x}{\left(\frac{\alpha}{\beta}x^{p-q}+1\right)^k}.
\label{eq:int_proof_1}
\end{equation}
Note that the term  $z = \left(\frac{\alpha}{\beta}x^{p-q}+1\right)^{-1}$ goes to $0$ when $x\rightarrow 0$ and to $1$ when $x\rightarrow +\infty$ if $p-q<0$. Furthermore, using $z$ as a variable change and by Definition~\ref{def:beta}, 
the integral~\eqref{eq:int_proof_1} becomes
\begin{align*}
\int_0^{+\infty} \frac{\text{d}x}{(\alpha x^p + \beta x^q)^k}&=\frac{\left(\frac{\alpha}{\beta}\right)^{m_p}}{\alpha^k(q-p)}\int_0^1 z^{m_p-1}(1-z)^{m_q-1}\text{d}z\\
&=\frac{B\left(m_p,m_q\right)}{\alpha^k(q-p)} \left(\frac{\alpha}{\beta}\right)^{m_p}\\
&= \frac{  \Gamma \left(m_p\right) \Gamma \left(m_q\right)}{\alpha^{k}\Gamma(k)(q-p)}\left(\frac{\alpha}{\beta}\right)^{m_p},
\end{align*}
concluding the proof.
\end{proof}




\end{document}